\documentclass{amsart}%{article}
\usepackage{exscale,amssymb,amsmath,bbm,color,amsmath,amsthm,amsfonts,stmaryrd}
\usepackage{graphicx}
\usepackage{bbm,exscale,stmaryrd}
\usepackage[foot]{amsaddr}
\usepackage{hyperref}

\newcommand{\E}[1]{\ensuremath{\mathbb{E} \left[#1 \right]}}
\newcommand{\Prob}[1]{\ensuremath{\mathbb{P} \left(#1 \right)}}

\newcommand{\I}[1]{\ensuremath{\mathbbm{1}_{ \{ #1 \} }}}
\newcommand{\R}{\ensuremath{\mathbb{R}}}

\newcommand{\convdist}{\ensuremath{\stackrel{d}{\longrightarrow}}}

\newcommand{\equidist}{\ensuremath{\stackrel{d}{=}}}

\renewcommand{\d}{\mathbbm{d}}
\newcommand{\m}{\mathbbm{m}}

\newcommand{\Del}[1]{\ensuremath{\Delta^{\!(\mathbf{#1})}}}

\newcommand{\fix}{\mathfrak{F}}

\newenvironment{itemize*}%
  {\vspace{-0.3cm}%
  \begin{itemize}%
    \setlength{\itemsep}{0pt}%
    \setlength{\parskip}{0pt}}%
  {\end{itemize}}

\newenvironment{enumerate*}%
  {\vspace{-0.3cm}%
  \begin{enumerate}%
    \setlength{\itemsep}{0pt}%
    \setlength{\parskip}{0pt}}%
  {\end{enumerate}}

\newtheorem{thm}{Theorem}[section]
\newtheorem{defn}[thm]{Definition}
\newtheorem{prop}[thm]{Proposition}

\begin{document}
\title[The Brownian CRT as a fixed point]{The Brownian continuum random tree as the unique solution to a fixed point equation}
\author{Marie Albenque}
\address{CNRS, LIX -- UMR 7161, \'Ecole Polytechnique, France.}
\author{Christina Goldschmidt }
\address{Department of Statistics and Lady Margaret Hall, University of Oxford, UK.}

\begin{abstract}
In this note, we provide a new characterization of Aldous' Brownian continuum random tree as the unique fixed point of a certain natural operation on continuum trees (which gives rise to a recursive distributional equation).  We also show that this fixed point is attractive.
\end{abstract}
\maketitle

\section*{Introduction}
\noindent The Brownian continuum random tree (BCRT), which was introduced and first studied by
Aldous~\cite{AldousCRT1, AldousCRT2, AldousCRT3}, is the
prototypical example of a random $\R$-tree/continuum random tree.  Its importance derives from the fact that it is the scaling limit of a large class of discrete trees including: all critical Galton--Watson trees with finite offspring variance~\cite{AldousCRT1,AldousCRT3}, unordered binary trees~\cite{MarckertMiermont}, uniform unordered trees~\cite{HaasMiermont2}, uniform unlabelled unrooted trees \cite{Stufler}, critical multitype Galton--Watson trees~\cite{Miermont} and random trees with a prescribed degree sequence satisfying certain conditions~\cite{BroutinMarckert}. It is also the scaling limit of random dissections~\cite{CurienHaasKortchemski} and random graphs from subcritical classes~\cite{PanagiotouStuflerWeller}.

Many of these convergence results are proved using some sort of functional coding.  However, particularly in the case of unordered trees, a natural functional coding whose distributional properties are easily understood is not always available.  In such settings, an alternative approach is desirable.

By a \emph{recursive distributional equation} for a random variable $X$ taking values in some Polish space $\mathcal{S}$, we mean an equation of the form
\begin{equation} \label{eqn:generalRDE}
X \equidist f((\xi_i, X_i), i \ge 1),
\end{equation}
where $X_1, X_2, \ldots$ are i.i.d.\ copies of $X$, independent of the family of random variables $(\xi_i)_{i \ge 1}$, and $f$ is a suitable $\mathcal{S}$-valued function.  We can, of course, think of this equation in terms of probability distributions: if $\mu$ is the distribution of $X$ and $F(\mu)$ is the distribution of the right-hand side then $\mu$ is a fixed point of the operator $F$.

For families of random variables which satisfy a natural recursive distributional equation, the so-called \emph{contraction method} has been demonstrated to be a powerful tool for proving convergence results.  Suppose that $(M_n)$ is a sequence of distributions for which we wish to prove that there exists a limit $M$.  The basic idea is as follows.  Suppose that $M_n$ can be described recursively in terms of $M_m$ for $m < n$.  This equation often allows one to guess a limiting version, in which $M$ is described in terms of itself.  In other words, $M$ should be the fixed point of some operator $\fix$.  Suppose that, in addition, $\fix$ is a contraction in a suitable metric on the space of probability measures.  Then Banach's fixed point theorem tells us that there exists a fixed point and that $M_n \to M$ as $n \to \infty$ in the sense of that metric.

This is straightforward in principle, but usually the recursive 
equation for $M_n$ does not have precisely the same form as the limiting operator.
Moreover, finding a metric in which $\fix$ is a contraction (but which also
yields weak convergence) is often highly non-trivial.  In practice, this method
has been applied very successfully for sequences of random variables (see, for
example, \cite{Rosler92,RoslerRuschendorf,NeiningerRuschendorf}), but so far there is only one result for the more complicated setting of convergence of stochastic processes \cite{RalphHenning}.

It is often the case that families of discrete trees have a recursive definition
or description.  Aldous~\cite{AldousRecursiveSelfSimilarity} proved that the BCRT is a fixed point for a
natural operation on continuum trees. With these two facts in mind, it is
natural to ask if a contraction method can be established for random trees.  This seems an
ambitious aim, and there are several technical issues to be overcome (not least
the choice of metric).  But our original motivation stems from the fact that, if
such a principle were to be established, then the characterization of possible
limits should be the first step.  In this article, we prove that the BCRT is the unique fixed point of an appropriate operator, and that this fixed point is attractive for a certain natural class of measures on continuum trees.

The rest of this note is organised as follows.  In Section 1, we provide an overview of the various definitions of the BCRT which already exist in the literature.  This also enables us to introduce various concepts we will need in the sequel.  We then set up our fixed point equation.  In Section 2, we prove that it has a unique solution. In Section 3, we show that repeatedly applying the fixed point operator to any suitable law on continuum trees gives convergence to the law of the BCRT in the sense of the Gromov--Prokhorov topology. Section 4 contains some concluding remarks.

\section{Overview of definitions of the BCRT}
\noindent We begin by introducing the notion of an $\R$-tree. 
\begin{defn}
A compact metric space $(T,d)$ is a \emph{$\R$-tree} if for all $x,y \in T$\vspace{-0.3cm}
\begin{itemize}
\item there exists a unique geodesic from $x$ to $y$ i.e.\ there exists a unique isometry $f_{x,y}: [0, d(x,y)] \to T$ such that $f_{x,y}(0) = x$ and $f_{x,y}(d(x,y)) = y$.  The image of $f_{x,y}$ is called $\llbracket x,y \rrbracket$;
\item the only non-self-intersecting path from $x$ to $y$ is $\llbracket x,y \rrbracket$ i.e.\ if $q: [0,1] \to T$ is continuous and injective and such that $q(0) = x$ and $q(1) = y$ then $q([0,1]) = \llbracket x,y \rrbracket$.
\end{itemize}
\end{defn}
An element $x \in T$ is called a \emph{vertex}.  A \emph{rooted $\R$-tree} is an $\R$-tree $(T,d)$ with a distinguished vertex $\rho$ called the \emph{root}.  The \emph{height} of a vertex $x$ is $d(\rho,x)$.  The \emph{degree} $\mathrm{deg}(x)$ of a vertex $x$ is the number of connected components of $T \setminus \{x\}$.  By a \emph{leaf}, we mean a vertex of degree 1; write $\mathcal{L}(T)$ for the set of leaves of $T$.  The tree $T$ is \emph{leaf-dense} if $T$ is the closure of $\mathcal{L}(T)$.  We will often want to endow an $\R$-tree with a Borel probability measure ($\mu$, say), which allows us to pick random points in the tree.

A \emph{measured metric space} $(X,d,\mu)$ is a complete metric space $(X,d)$ equipped with a Borel probability measure $\mu$ (with respect to the metric $d$) on $X$.  Define a first equivalence relation by declaring two such spaces $(X,d,\mu)$ and $(X',d',\mu')$ to be \emph{GHP-equivalent} if there exists an isometry $f: X \to X'$ such that the image of $\mu$ under $f$ is $\mu'$.  Let $\mathcal{S}$ denote the space of GHP-equivalence classes of compact measured metric spaces.  Then $\mathcal{S}$ is Polish when endowed with the Gromov--Hausdorff--Prokhorov topology \cite{AbrahamDelmasHoscheit}.  Define a second equivalence relation by declaring $(X,d,\mu)$ and $(X',d',\mu')$ to be \emph{GP-equivalent} if there exists an isometry $g: \text{supp}(\mu) \to X'$ such that the image of $\mu$ under $g$ is $\mu'$, where $\text{supp}(\mu)$ denotes the topological support of $\mu$.  Let $\mathcal{S}'$ denote the space of GP-equivalence classes of compact measured metric spaces.  Then $\mathcal{S}'$ is Polish when endowed with the Gromov--Prokhorov topology \cite{GrevenPfaffelhuberWinter}.

\subsection{The BCRT as an $\R$-tree encoded by a Brownian excursion}
\noindent A standard way to generate $\R$-trees is via functional encoding. Suppose that $h: [0,\infty) \to [0, \infty)$ is a continuous function of compact support such that $h(0) = 0$.  Use it to define a pseudo-metric $\tilde{d}$ by
\[
\tilde{d}(x,y) = h(x) + h(y) - 2 \inf_{x\wedge y \leq t \leq x \vee y} h(t), \qquad x,y \ge 0.
\]
Define an equivalence relation $\sim$ by letting $x \sim y$ if $\tilde{d}(x,y) = 0$.  Let $T = [0,\infty)/\sim$, denote by $\tau: [0,\infty) \to T$ the canonical projection and let $d$ be the metric induced on $T$ by $\tilde{d}$.  If $\sigma$ is the supremum of the support of $h$ then note that $\tau(s) = 0$ for all $s \geq \sigma$.  This entails that $T = \tau([0,\sigma])$ is compact.  The metric space $(T,d)$ can then be shown to be an $\R$-tree (see Le Gall~\cite{LeGallSurvey}).  
The tree $T$ can be naturally rooted at $\rho = \tau(0)$, the equivalence class of 0, and we will sometimes think of it as a rooted object and sometimes not. There is a natural measure $\mu$ on $T$ given by the push-forward of the uniform distribution on $[0,\sigma]$ under the projection $\tau$.

We define the BCRT $(\mathbb{T},\d)$ to be the $\R$-tree encoded by 
\[
h(t) = \begin{cases}
          2e(t) & 0 \leq t \leq 1, \\
          0 & t > 1,
          \end{cases}
\]
where $(e(t), 0 \le t \le 1)$ is a standard Brownian excursion. We usually endow $(\mathbb{T},\d)$ with the probability measure $\m$ which is the push-forward of the Lebesgue measure on $[0,1]$.  

\subsection{The BCRT as a limit of discrete trees}
\noindent Let $T_n$ be the ordered rooted tree representing the genealogy of a Galton--Watson branching process with offspring distribution having mean 1 and variance $\sigma^2 \in (0, \infty)$.  Think of $T_n$ as a metric space by endowing it with the graph distance $d_{\mathrm{gr}}$ (which puts neighbouring vertices at distance 1). Let $\mu_{n}$ be the uniform measure on the vertices of $T_{n}$. Then
\[
(T_n, n^{-1/2}d_{\mathrm{gr}},\mu_{n})\convdist (\mathbb{T},\sigma^{-1}\d,\m)
\]
as $n \to \infty$, in the Gromov--Hausdorff--Prokhorov sense.  (The convergence in distribution is originally due to Aldous \cite{AldousCRT1}, although this formulation is closer to that of Le Gall \cite{LeGall06}.)

\subsection{The BCRT via random finite-dimensional distributions}\label{sec:findim}
\noindent We may also characterize the BCRT as the unique continuum random tree having certain distributional properties.  We must first introduce properly what we mean by a continuum tree.

\begin{defn}
A \emph{continuum tree} is a triple $(T,d, \mu)$ where $(T,d)$ is an (unrooted) $\R$-tree and $\mu$ is a Borel probability measure on $T$ which is non-atomic and satisfies
\begin{itemize}
\item $\mu(\mathcal{L}(T)) = 1$;
\item for every $v \in T$ of degree $k = \mathrm{deg}(v) \ge 2$, let $T_1, \ldots, T_k$ be the connected components of $T \setminus \{v\}$; then $\mu(T_i) > 0$ for all $1 \le i \le k$.
\end{itemize}
\end{defn}

The set of continuum trees can naturally be endowed with the Gromov-Hausdorff-Prokhorov topology, as briefly discussed at the beginning of this section.

\begin{defn}
A \emph{continuum random tree} (CRT) is a random variable taking values in the set of continuum trees.
\end{defn}

(In \cite{AldousCRT3}, Aldous makes slightly different definitions of these quantities which, in particular, use rooted trees and, hence, the \emph{pointed} Gromov-Hausdorff-Prokhorov topology). It will be important in the sequel to observe that, if we consider the BCRT to be rooted at the equivalence class of 0 in the Brownian excursion construction, then the root has the same distribution as a uniform pick from $\m$ on $\mathbb{T}$.

Given a CRT $(T,d,\mu)$, let $V_1, V_2, \ldots$ be i.i.d.\ samples
from the measure $\mu$.  For $m \ge 2$, define the \emph{reduced tree}
$\mathcal{R}(m)$ to be the subtree of $T$ spanned by $V_1, V_2, \ldots, V_m$.
For every $m \ge 2$, $\mathcal{R}(m)$ is a discrete tree with edge-lengths and
labelled leaves, and so its distribution is specified by its \emph{tree-shape},
$\mathbf{t}$, an unrooted tree with $m$ labelled leaves, and its \emph{edge-lengths}. 
The reduced trees are clearly \emph{consistent}, in that $\mathcal{R}(m)$ is a subtree of $\mathcal{R}(m+1)$.

\begin{thm}[Aldous~\cite{AldousCRT3}] \label{th:rfdds}
The distribution of a CRT $(T,d,\mu)$ is specified entirely by its \emph{random finite dimensional distributions}, that is, the distribution of $\mathcal{R}(m)$ for all $m \ge 2$.
\end{thm}

The reduced trees of the BCRT are binary almost surely.  This entails that $\mathcal{R}(m)$ has $2m-2$ vertices and $2m-3$ edges. Let $\mathbf{t}$ be its tree-shape and $x_1, x_2, \ldots, x_{2m-3}$ be its edge-lengths listed in any (arbitrary, but fixed) order.  Then $\mathcal{R}(m)$ has density
\begin{equation}\label{eq:findimcrt}
f(\mathbf{t}; x_1, x_2, \ldots, x_{2m-3}) = \left( \sum_{i=1}^{2m-3} x_i \right) \exp \left( - \frac{1}{2} \left( \sum_{i=1}^{2m-3} x_i \right)^2 \right).
\end{equation}
Note that this implies that the tree-shape is, in fact, uniform on the set of binary tree-shapes with $m$ labelled leaves, and that the edge-lengths have an exchangeable distribution.
  We observe, for future reference, that the distance between two uniformly-chosen points of the BCRT has the Rayleigh distribution, with density $xe^{-x^2/2}$ and expectation $\sqrt{\pi/2}$.

(Note that  in \cite{AldousCRT3}, Aldous restricts his discussion to binary trees, but the theory is easily extended; see Haas and Miermont~\cite{HaasMiermont1}.)

\subsection{The BCRT as a fixed point}\label{sub:BCRTFix}
\noindent The principal contribution of this paper is a characterization of the BCRT as the unique fixed point of a certain operation on CRT's.  We need a couple of notational ingredients.   We first recall the definition of the Dirichlet distribution.

\begin{defn}
Let $\alpha_1, \alpha_2, \ldots, \alpha_n > 0$. A random variable taking values in the space $\{\mathbf{s} = (s_1, s_2, \ldots, s_n): s_i \ge 0, 1 \le i \le n, \sum_{i=1}^n s_i = 1\}$ has the \emph{Dirichlet distribution} with parameters $(\alpha_1, \alpha_2, \ldots, \alpha_n)$
(written $\mathrm{Dir}(\alpha_1, \alpha_2, \ldots, \alpha_n)$) if it has density
\[
\frac{\Gamma(\sum_{i=1}^n \alpha_i)}{\prod_{i=1}^n \Gamma(\alpha_i)}
x_1^{\alpha_1-1} x_2^{\alpha_2 -1} \ldots x_{n-1}^{\alpha_{n-1}-1} (1 - x_1 -
x_2 - \cdots - x_{n-1})^{\alpha_n-1}
\]
with respect to $(n-1)$-dimensional Lebesgue measure.
\end{defn}

Let $\mathcal{M}$ be the set of probability distributions on (GHP-equivalence classes of) measured $\R$-trees.  We define $\fix: \mathcal{M} \to \mathcal{M}$ as follows: for $M \in \mathcal{M}$,
\begin{itemize}
\item  Sample independent trees $(T_1, d_1, \mu_1)$, $(T_2, d_2, \mu_2)$, $(T_3, d_3, \mu_3)$ having distribution $M$;
\item For $1 \le i \le 3$, pick a vertex $X_i \in T_i$ according to the measure $\mu_i$;
\item Sample $\Delta = (\Delta_1, \Delta_2, \Delta_3) \sim \mathrm{Dir}(1/2,1/2,1/2)$ independently; 
\item Rescale the trees to obtain $(T_1, \Delta_1^{1/2} d_1, \Delta_1 \mu_1)$, $(T_2, \Delta_2^{1/2} d_2, \Delta_2 \mu_2)$, $(T_3, \Delta_3^{1/2} d_3, \Delta_3 \mu_3)$;
\item Identify the vertices $X_1$, $X_2$ and $X_3$ in the rescaled trees to obtain a single larger tree $(T^{\circ}, d)$ with a marked branch-point; the three measures $\Delta_1 \mu_1$, $\Delta_2 \mu_2$ and $\Delta_3 \mu_3$ naturally give rise to a (probability) measure $\mu$ on $T^{\circ}$;
\item Forget the marked branch-point in order to obtain $(T, d, \mu)$; $\fix(M)$ is the distribution of $(T, d, \mu)$.
\end{itemize}

\begin{figure}[t]
\begin{center}
\includegraphics[scale=0.8,page=2]{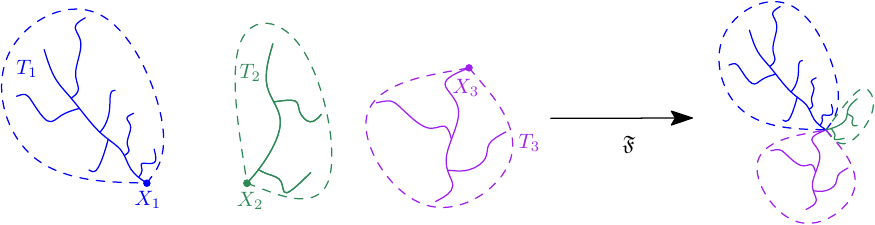}
\end{center}
\caption{The operator $\fix$}\label{fig:dectree}
\end{figure}

The operation on trees given by the function $\fix$ was first described by Aldous~\cite{AldousRecursiveSelfSimilarity}.  Let $\mathbb{M}$ be the law of the BCRT.  Theorem 2 of \cite{AldousRecursiveSelfSimilarity} implies, when rephrased in our terms, that $\mathbb{M}$ is a solution of $M = \fix(M)$. Actually, what is shown in~\cite{AldousRecursiveSelfSimilarity} is the following statement of the ``reverse'' of this construction: take a BCRT $(\mathbb{T}, \d, \m)$ and pick three points independently according to $\m$; the paths between pairs of these points intersect in a unique branch-point.  

Splitting at this branch-point then gives three BCRT's, which have been randomly rescaled by $(\Delta_1, \Delta_2, \Delta_3)$ and depend on one another only through this rescaling. Moreover the former branch-point yields a point chosen independently from the mass measure of each of the three subtrees. (An expanded proof of Aldous' Theorem 2 may be found in \cite{ABBG2}.) We will comment on this reversed perspective at the end of the paper.

Let $M$ be a solution to the fixed point equation.  Write $(\Omega, \mathcal{F}, \mathbb{P})$ for the probability space on which all the forthcoming random objects are defined.  In particular, under $\mathbb{P}$, let $(T,d,\mu)$ be a continuum random tree sampled from the distribution $M$.  

The first main result of this article is the following theorem, which is proved in the next section.
\begin{thm}\label{th:fxpt}
Suppose that $M$ is a law on continuum trees which is a fixed point of $\fix$.  Then there exists $\alpha>0$ such that if $(T,d,\mu)$ is sampled according to $M$ then $(T,\alpha d,\mu)$ has the law of the BCRT.
\end{thm}

Before going further, we will briefly discuss the requirement that $M$ be a measure on continuum trees. Let $(T,d,\mu)$ be sampled according to $M$. The assumption that $\mu$ is is carried by the leaves of $T$ ensures that any fixed point of $\mathfrak{F}$ is binary. Indeed, if $\mu$ gives positive mass to $T \setminus \mathcal{L}(T)$ then it is clear that we can create non-binary branch-points. Given that $M$ is a fixed point of $\mathfrak{F}$ and that $\mu$ is carried by $\mathcal{L}(T)$, $\mu$ cannot, in fact, be atomic.  Indeed, suppose (for a contradiction) that there exists $x \in \mathcal{L}(T)$ such that $\mu(\{x\})>0$. Then, with positive probability, $\mathfrak{F}$ creates a tree which carries positive mass at a non-leaf, contradicting $\mu(T \setminus \mathcal{L}(T)) = 0$.

We now discuss the assumption that $\mu$ has to give a positive measure to any connected subcomponent of the tree. Recall that the BCRT $(\mathbb{T},\d,\m)$ is encoded by $(2e(t), 0 \le t \le 1)$, where $e$ is a standard Brownian excursion.  
Consider an independent Poisson point
process (PPP) on $[0,1]\times[0,\infty)$ with intensity $ds \otimes x^{-3}dx$. For each point $(s,x)$ of the PPP graft a massless
branch of length $x$ to the point of $\mathbb{T}$ corresponding to $s$ under the canonical projection $\tau$ (note that $\tau(s)$ is almost surely a leaf). As there are almost surely only finitely many of these branches having length longer than any $\epsilon > 0$, this
construction yields a compact metric space and, therefore, induces a probability distribution on the set of measured $\R$-trees. A
simple computation shows that this distribution is a solution of the fixed point
equation which is clearly not isometric to the BCRT. However, it seems reasonable to want to exclude such non-continuum tree-valued solutions.

Our second main result is as follows.

\begin{thm} \label{th:cvg}
Suppose that $M$ is a law on continuum trees such that if $(T,d,\mu) \sim M$ and, given $(T,d,\mu)$, $V_1, V_2$ are sampled independently from $\mu$, then $\E{d(V_1,V_2)}$ exists and is equal to $\sqrt{\pi/2}$.  Let $M_n = \mathfrak{F}^n M$.  Then
\[
M_n \to \mathbb{M}
\]
as $n \to \infty$, in the sense of weak convergence of measures with the Gromov--Prokhorov topology.
\end{thm}

Note that if $\E{d(V_1,V_2)} = \alpha^{-1} \sqrt{\pi/2}$ for some $\alpha \neq 1$ then the same result holds on multiplying the metric $d$ by $\alpha$. We emphasize that  there is no need for $M$ to be a law on \emph{binary} continuum trees. For example, $M$ could be the law of a stable tree of parameter in $(1,2)$ (which has only infinitary branch-points, almost surely). Theorem~\ref{th:cvg} is proved in Section 3.

\section{Uniqueness of the fixed point: proof of Theorem~\ref{th:fxpt}}

\noindent We will prove Theorem~\ref{th:fxpt} via random finite-dimensional distributions and Theorem~\ref{th:rfdds}. We start by thinking about the distance between two uniformly-chosen points. Throughout this section, we suppose that $M$ is a measure on continuum trees which is a fixed point of $\fix$. We write $\mathcal{S}(m), m \ge 2$ for the reduced trees of a tree $(T,d,\mu)$ sampled according to $M$.

\subsection{Two-point distances}
Suppose that $(T,d,\mu)$ is sampled from $M$ and let $D$ be the distance between two points of $(T,d)$ sampled independently according to $\mu$. 

\begin{prop} \label{prop:distance}
There exists a constant $\alpha>0$ such that $\alpha D$ has the Rayleigh distribution.
\end{prop}

\begin{proof}
Suppose that $(T_1, d_1, \mu_1)$, $(T_2, d_2, \mu_2)$ and $(T_3, d_3, \mu_3)$ are sampled independently from $M$.  Apply $\fix$ with $\Delta = (\Delta_1, \Delta_2, \Delta_3)$ to obtain a new tree $(T,d,\mu) \sim M$.  Suppose now that we sample two points independently according to $\mu$. Let $P_1$, $P_2$ and $P_3$ be the number of these points falling in the subtrees of $T$ corresponding to $T_1$, $T_2$ and $T_3$ respectively.  Then, conditional on $\Delta$, we have $(P_1, P_2, P_3) \sim \mathrm{Multinomial}(2; \Delta_1, \Delta_2, \Delta_3)$.  
Let $D$ be the distance between the two points. Then
\begin{equation} \label{eq:Rayleigh}
D = \sqrt{\Delta_1} D_1\I{P_1 > 0} + \sqrt{\Delta_2} D_2\I{P_2 > 0} + \sqrt{\Delta_3}D_3\I{P_3 > 0},
\end{equation}
where $D_1$, $D_2$ and $D_3$ are three independent copies of $D$, independent of everything else on the right-hand side, corresponding to the distances between two uniformly-chosen points in each of the three  subtrees.  Let $W_k = \sqrt{\Delta_k}\I{P_k > 0}$, $k=1,2,3$.  Then this is precisely the setting of the smoothing transform studied by Durrett and Liggett~\cite{DurrettLiggett}.  In that paper, it is shown that the nature of the family of solutions to such distributional fixed point equations depends on the analytic properties of a certain function depending on the moments of $W_1, W_2, W_3$: for $s \ge 0$, let
\[
\nu(s) = \log \left( \sum_{k=1}^3 \E{W_k^{s} \I{W_k > 0}} \right).
\]
By symmetry, $\nu(s) = \log \left(3 \E{W_1^{s} \I{W_1 > 0}} \right)$.
Now, $\Prob{P_1 > 0 | \Delta_1} = \Delta_1^2 + 2 \Delta_1(1 - \Delta_1) = 2 \Delta_1 - \Delta_1^2$.  Since $\Delta_1 \sim \mathrm{Beta}(1/2,1)$, we obtain
\[
\E{W_1^{s} \I{W_1 > 0}} = 2 \E{\Delta_1^{\frac{s}{2}+1}} - \E{\Delta_1^{\frac{s}{2}+2}} = \frac{2}{s + 3} - \frac{1}{s + 5}.
\]
Hence,
\[
\nu(s) = \log \left( \frac{3(s + 7)}{(s + 3)(s + 5)} \right),
\]
which is finite for all $s \ge 0$ and has its unique zero in $s \ge 0$ at $s = 1$.  Moreover, $\nu'(1) = -7/24 < 0$.  Theorems 1 and 2 of \cite{DurrettLiggett} then entail that the equation (\ref{eq:Rayleigh}) has a unique fixed point, up to a constant scaling factor.  Finally, the distance between two uniformly chosen points in a BCRT has the Rayleigh distribution and that must be a solution to (\ref{eq:Rayleigh}). Define $\alpha$ by the relation $\E{D}  = \alpha^{-1} \sqrt{\pi/2}$. Since the Rayleigh distribution has mean $\sqrt{\pi/2}$, this concludes the proof.
\end{proof}

For future reference, we write $\mathfrak{F}_{\mathrm{sm}}$ for the operator which takes the law of a non-negative real-valued random variable $D'$ and associates to it the law of 
\[
 W_1 D'_1 + W_2 D'_2 + W_3 D'_3
\]
where $D_1'$, $D_2'$ and $D_3'$ are three independent copies of $D'$, independent of everything else on the right-hand side, and where $W_1, W_2, W_3$ are exactly as above.

\subsection{A coupling} \label{sub:coupling}

\noindent Having determined the distribution of $\mathcal{S}(2)$ (which, of course, has trivial tree-shape), we now want to determine the distribution of the reduced trees $\mathcal{S}(m), m \ge 3$. In order to do so, we proceed by coupling a tree $T$ distributed according to $M$ and a realisation $\tilde T$ of the BCRT, using the operator $\fix$.  We will, in fact, find it convenient to set up this coupling more generally.  Indeed, fix $n \ge 0$ and let $M'$ be a general law on continuum trees (which is not necessarily a fixed point of $\mathfrak{F}$). Now let $T \sim \mathfrak{F}^{n+1} M'$; we will produce a coupling of $T$ and $\tilde{T}$.  

Before we can describe this coupling, we need to establish some notation.
For $n \ge 0$, let $\Sigma = \cup_{i=0}^{\infty}\{1,2,3\}^i$ be the set of words on the alphabet $\{1,2,3\}$ where, by convention, $\{1,2,3\}^0 = \{\emptyset\}$ is the set containing the empty word. Let $\Sigma_n=\cup_{i=0}^n\{1,2,3\}^i$ be the set of words with at most $n$ letters.  For $\mathbf{i} \in \Sigma$, write $|\mathbf{i}|$ for the length of the word $\mathbf{i}$.  For $1 \le m \le |\mathbf{i}|$, write $\mathbf{i}_m$ for the $m$th letter of $\mathbf{i}$ and $\mathbf{i}_{[m]} = \mathbf{i}_1 \ldots \mathbf{i}_m$ for the prefix consisting of the first $m$ letters of $\mathbf{i}$.  

Fix $n \ge 0$ and start from a family $(T_{\mathbf{i}})_{\mathbf{i} \in \{1,2,3\}^{n+1}}$ of $3^{n+1}$ independent continuum random trees with common law $M'$, and a family $(\tilde T_{\mathbf{i}})_{\mathbf{i} \in \{1,2,3\}^{n+1}}$ of $3^{n+1}$ independent copies of the BCRT.  We will refer to these as the \emph{input trees} and will use them and successive applications of $\fix$ in order to build the trees $T = T_{\emptyset}$ and $\tilde{T} = \tilde{T}_{\emptyset}$.  At each application of $\fix$, we will use the same scaling factors and glue together subtrees with the same labels. More precisely, let $(\Del i)_{\mathbf{i}\in \Sigma}$ and $(U^{(\mathbf{i})})_{\mathbf{i}\in \Sigma}$ be independent families of independent random variables where, for each $\mathbf{i}\in \Sigma$, $\Del i = (\Del i _1, \Del i _2, \Del i _3) \sim \mathrm{Dir}(1/2,1/2,1/2)$ and $U^{(\mathbf{i})} = (U^{(\mathbf{i})}_1,U^{(\mathbf{i})}_2,U^{(\mathbf{i})}_3)$, where $U^{(\mathbf{i})}_1$, $U^{(\mathbf{i})}_2$ and $U^{(\mathbf{i})}_3$ are independent uniform random variables on $[0,1]$. 

\begin{figure}[t]
\begin{center}
\includegraphics[scale=0.95]{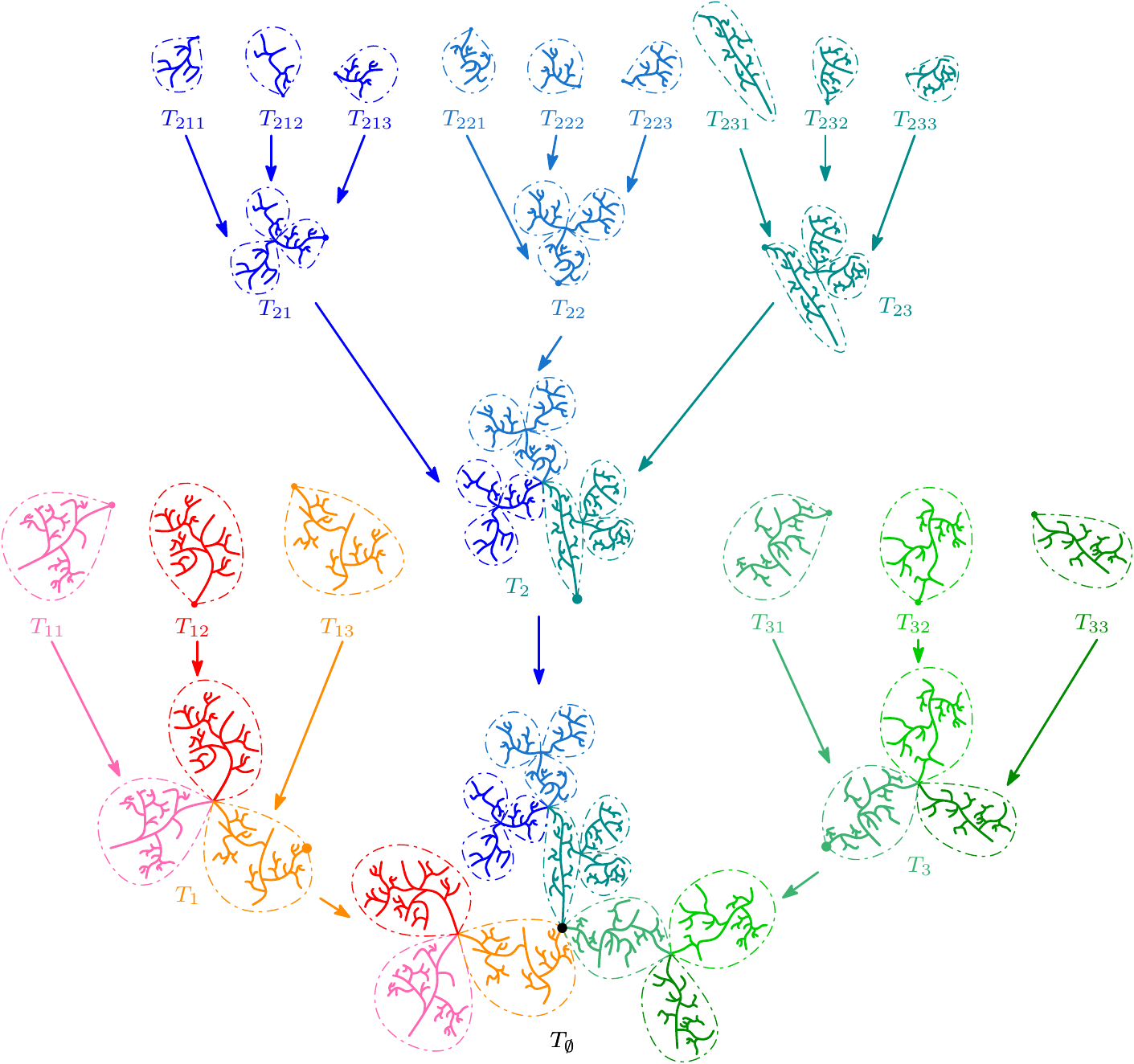}
\end{center}
\caption{Example of the construction of $T_{\emptyset}$ from some input trees (the rescaling is omitted here). For instance, here, $L^{(2,1,3)}=3$, $L^{(2,2,3)}=1$, $L^{(2,3,3)}=1$ and $L^{(\emptyset,2,3)}=31$. \label{fig:labelsTrees}}
\end{figure}

The families $(T_{\mathbf{i}})_{\mathbf{i} \in \Sigma_{n}}$ and $(\tilde T_{\mathbf{i}})_{\mathbf{i} \in \Sigma_{n}}$ are constructed recursively as follows. The tree $T_{\mathbf i}$ (resp. $\tilde T_{\mathbf i}$) is constructed by applying $\fix$ to $T_{\mathbf i 1}$, $T_{\mathbf i 2}$ and $T_{\mathbf i 3}$ (resp. $\tilde T_{\mathbf i 1}$, $\tilde T_{\mathbf i 2}$ and $\tilde T_{\mathbf i 3}$) with scaling factors $\Del i _1$, $\Del i _2$ and $\Del i _3$, where we emphasize that the \emph{same} scaling factors are used to construct both families. In each of the trees $T_{\mathbf{i}1}$, $T_{\mathbf{i}2}$ and $T_{\mathbf{i}3}$ (resp. $\tilde T_{\mathbf{i}1}$, $\tilde T_{\mathbf{i}2}$ and $\tilde T_{\mathbf{i}3}$), we need to pick a uniform point which tells us where to glue them together (once rescaled) to form $T_{\mathbf{i}}$ (resp. $\tilde T_{\mathbf{i}}$).  But if $|\mathbf{i}| < n$, we will also want to keep track of where these uniform points sit in the trees at level $n+1$.  We can split this problem into two parts: first finding the label of the subtree at level $n$ in which a particular uniform point lies, and then finding where precisely within that subtree it sits.  We will use the random variables $(U^{(\mathbf i)})_{\mathbf i \in \Sigma_{n}}$ to determine the label of the subtree, and the exact location of the point is then a uniform pick from that subtree.  We will use the same labels in $T$ and $\tilde T$ but independent picks from the respective subtrees chosen.

Let $\Delta_k=\Del \emptyset _k$, for $1\leq k\leq 3$. For $\mathbf{i} \in \Sigma_n \setminus \{\emptyset\}$ and $1 \le k \le 3$, recursively define $\Delta_{\mathbf i k} := \Delta_{\mathbf i}\Del i _k$. 
In addition, for $\mathbf{j} \in \Sigma$, write $\Delta_{\mathbf{j}}^{\!(\mathbf{i})} :=\Delta_{\mathbf{i} \mathbf{j}} / \Delta_{\mathbf{i}}$.  
For $\mathbf{i}\in \Sigma_{n}$, let $\mathbf{j}\in\Sigma_{n-|\mathbf{i}|+1}$. By construction, $T_\mathbf{i}$ has a subtree which is equal to $T_\mathbf{ij}$, up to rescaling $\mu^{(i)}_{T_\mathbf{ij}}$ and $d_{T_\mathbf{ij}}$ by $\Delta_{\mathbf j}^{(\mathbf{i})}$ and $\sqrt{\Delta_{\mathbf j}^{(\mathbf{i})}}$ respectively. When the context is clear, we ignore the rescaling and refer to this subtree as $T_\mathbf{ij}$. Then the probability that a uniform point in $T_{\mathbf{i}}$ belongs to the subtree $T_{\mathbf{ij}}$ is equal to $\Del i_{\mathbf{j}}$. 
 For $\mathbf i\in \Sigma_n$ and $1\leq k\leq 3$, we define a word $L^{(\mathbf{i},k,n+1)}$ of length $n-|\mathbf i|$ such that $\mathbf{i}kL^{(\mathbf{i},k, n+1)}$ represents the label of the input tree at level $n+1$ in which the uniform point sampled in $T_{\mathbf{i}k}$ sits. It is convenient to use a random recursive partition of the interval $[0,1]$ to choose this point. The left boundaries of the intervals of this partition are defined by
\[
B^{(\mathbf{i}, k)}_1 = 0, \quad B^{(\mathbf{i}, k)}_2 = \Delta^{(\mathbf{i}k)}_1, \quad B^{(\mathbf{i}, k)}_3 = \Delta^{(\mathbf{i}k)}_1 + \Delta^{(\mathbf{i}k)}_2.
\]
Recursively, for $\mathbf{j} \in \Sigma_{n-|\mathbf{i}|-1}\backslash \{\emptyset\}$, let
\[
B^{(\mathbf{i}, k)}_{\mathbf{j}1} = B^{(\mathbf{i}, k)}_{\mathbf{j}}, \quad B^{(\mathbf{i}, k)}_{\mathbf{j}2} = B^{(\mathbf{i}, k)}_{\mathbf{j}} + \Delta^{(\mathbf{i}k)}_{\mathbf{j}1}, \quad B^{(\mathbf{i}, k)}_{\mathbf{j}3} = B^{(\mathbf{i}, k)}_{\mathbf{j}} + \Delta^{(\mathbf{i}k)}_{\mathbf{j}1} + \Delta^{(\mathbf{i}k)}_{\mathbf{j}2}.
\]
For $0\leq  \ell \leq  n-|\mathbf{i}|-1$, if $L_{[\ell]}^{(\mathbf{i},k, n+1)} = \mathbf{j}$ then let
\[
L_{\ell+1}^{(\mathbf{i}, k, n+1)} = \begin{cases}
1  & \text{ if $B^{(\mathbf{i}, k)}_{\mathbf{j}1} \le U^{(\mathbf{i})}_ k \le B^{(\mathbf{i}, k)}_{\mathbf{j}2}$}\\
2 & \text{ if $B^{(\mathbf{i}, k)}_{\mathbf{j}2} < U^{(\mathbf{i})}_ k \le B^{(\mathbf{i}, k)}_{\mathbf{j}3}$} \\
3 & \text{ if $B^{(\mathbf{i}, k)}_{\mathbf{j}3} < U^{(\mathbf{i})}_ k$.}
\end{cases}
\]
Observe that the definition of $L^{(\mathbf{i},k,n+1)}$ depends only on $(\Delta^{(\mathbf{i})})_{\mathbf{i} \in \Sigma_n}$ and $(U^{(\mathbf{i})})_{\mathbf{i} \in \Sigma_n}$. 

So, finally, when we sample the uniform point in $T_{\mathbf{i}k}$ (resp. in $\tilde T_{\mathbf{i}k})$ needed to create $T_{\mathbf{i}}$ (resp. $\tilde T_{\mathbf{i}})$, the value of $L^{(\mathbf i,k,n+1)}$ gives the index of the input tree in which the point sits. Then, conditionally on this choice, we sample the point uniformly from $T_{\mathbf{i}kL^{(\mathbf i,k, n+1)}}$ (resp. $\tilde T_{\mathbf{i}kL^{(\mathbf i,k, n+1)}}$).

Certain statistics of the trees constructed by this coupling depend \emph{only} on the scaling factors and not on the input trees.  These statistics are identical for the two trees.  Moreover, because the construction can be performed consistently for different values of $n$, we can make sense of an infinite version of it as a projective limit, which results in a family $(L^{(\mathbf{i},k,n+1)}, \mathbf{i} \in \Sigma_n, 1 \le k \le 3, n \ge 0)$ of labels which encode the gluing points all the way down.

\subsection{The reduced trees}
\noindent Now, consider $\mathcal{S}(3)$.  Again, in this case, the tree-shape is deterministic.  We will show that the lengths of the three branches can each be expressed as sums of rescaled distances between pairs of uniform points. Fix $n$ and consider $(T, d,\mu) \sim M$ to be constructed as in the previous section after recursive applications of $\fix$ to level $n+1$. We sample three new independent uniform points from $T$. We wish to determine whether their branch-point in $T$ has been used as a gluing point in the construction of $T$ and, if so, at which step of the construction.

 If, when we decompose $T$ into its three subtrees $T_1$, $T_2$ and $T_3$, the three new points all happen to fall into different subtrees, then their branch-point is determined and is the point $G$ used to glue $T_1$, $T_2$ and $T_3$ together. However, if at least two points fall into the same subtree, say $T_1$, we must then further decompose $T_1$ in order to try to determine the location of the branch-point. We continue this process recursively until either (a) the three points all fall into different subtrees or (b) we reach level $n+1$. Now observe that the probability that the points are separated depends only on the \emph{sizes} of the subtrees and not on the underlying structure of the trees. In particular, this means that we can use the infinite version of our coupling.  So let $N_3$ be the smallest value $k \ge 1$ such that the branch-point between our three uniform points is a gluing point at level $k$ in the infinite coupling. More generally, let $N_m$ be the smallest value $k \ge 1$ such that the branch-points between our $m$ uniform points are all determined as gluing points at levels at most $k$.

\begin{prop} \label{prop:geom}
For $m \ge 3$, $N_m < \infty$ almost surely.  In particular, $\Prob{N_3 =k} = \frac{2}{35} \left( \frac{33}{35} \right)^{k-1}$, for $k \ge 1$.
\end{prop}

\begin{proof}
We proceed by induction on $m$ and start with the case $m=3$. There are three possibilities for the way in which the three points are distributed amongst the subtrees $T_1$, $T_2$ and $T_3$:
\begin{enumerate}
\item The three points fall in different subtrees.
\item All three points fall in the same subtree.
\item Two points fall in the same subtree and the remaining point falls in a different subtree.
\end{enumerate}
In case (1), as observed above, the branch-point is necessarily $G$.  In case (2), we have a new independent copy of the original problem of finding the branch-point between three points chosen uniformly from a copy of $T$.  In case (3), the branch-point we seek is the same as that between $G$ and the two uniform points which fell in the same subtree.  But $G$ is also a uniformly chosen point in that subtree.  So again, it remains to find the branch-point between three points chosen uniformly from a copy of $T$.  Indeed, unless case (1) occurs, we recursively obtain a new (independent) copy of the original problem. (See Figure~\ref{fig:recdepth} for an illustration.) Since case (1) occurs with strictly positive probability, it follows that $N_3$ is a geometric random variable. The probability that the three points fall in different subtrees at any step is given by
\[
6 \E{\Delta_1 \Delta_2 \Delta_3} = \frac{2}{35},
\]
and so we obtain $\Prob{N^{(n)}_3 =k} = \frac{2}{35} \left( \frac{33}{35} \right)^{k-1}$, for $k \ge 1$.

For $m\geq 4$, we proceed by induction. It will be convenient to define $N_2 = 0$.  Suppose that $N_{\ell} < \infty$ almost surely for $3 \le \ell \le m-1$.  There are again three possibilities for the distribution of $m$ uniform points amongst the subtrees $T_1$, $T_2$ and $T_3$:
\begin{enumerate}
\item At least two points fall in different subtrees from the rest.
\item All $m$ points fall in the same subtree.
\item $m-1$ points fall in the same subtree and the remaining point falls in a different subtree.
\end{enumerate}
In cases (2) and (3), we obtain again a new copy of the same problem. In case (1), we get two or three independent copies of a problem of strictly smaller size.  Again, case (1) occurs with strictly positive probability at each level, and so we have a geometric number of trials, $\tilde{N}_m$ say, until it does. Then $\tilde{N}_m < \infty$ almost surely. On $\{\tilde{N}_m < \infty\}$, there are random variables $A_1, A_2, A_3$ such that $2 \le A_2 \le A_1 \le m-2$, $0 \le A_3 \le A_2$ and $A_1 + A_2 + A_3 = m$, which represent the numbers of points falling in different subtrees (in decreasing order).  Then the remaining number of levels we have to explore in order to separate all of the points has the same distribution as 
\[
\max \left\{ N_{A_1+1}, N_{A_2+1}, N_{A_3+1} \right\},
\]
where the three random variables in the maximum are conditionally independent given $A_1, A_2, A_3$. Since $A_3+1 \le A_2+1 \le A_1+1 \le m-1$, it follows straightforwardly that $N_m < \infty$ almost surely. The result then follows by induction on $m$.

\end{proof}

\begin{figure}[t]
\begin{center}
\includegraphics[scale=1,page=2]{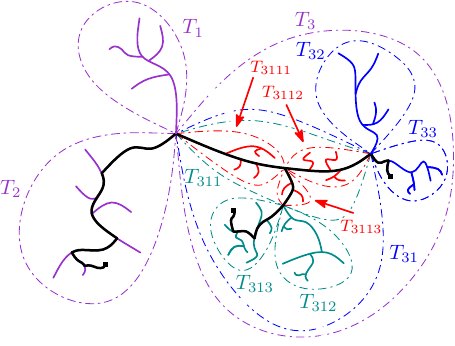}
\end{center}
\caption{Finding the branch-point between three uniform points. One point falls in $T_2$ and two fall in $T_3$, so we must further decompose $T_3$.  We have an independent copy of the original problem in $T_3$ (where one of the points considered is now $G$).  One of the points falls in $T_{33}$ and the other two in $T_{31}$, so we must further decompose $T_{31}$.  One of the points now falls in $T_{313}$ but the two others are in $T_{311}$, so we repeat in $T_{311}$. Finally, the three points fall in different subtrees of $T_{311}$ and so we obtain $N_3 = 4$.  \label{fig:recdepth}}
\end{figure}

\begin{prop} \label{prop:redtree}
$(\mathcal{S}(m), m \ge 2)$ have the same joint distribution as the reduced trees $(\mathcal{R}(m), m \ge 2)$ of the BCRT.
\end{prop}

\begin{proof}
Fix $\epsilon > 0$ and $m \ge 2$.  By Proposition~\ref{prop:geom}, there exists $n$ sufficiently large that we have $\Prob{N_m \le n+1} > 1 - \epsilon$. Consider the trees $T$ and $\tilde T$ constructed by the above coupling to recursion depth $n+1$, so that $T$ is distributed according to $M$ and $\tilde T$ according to the law of the BCRT. Consider $m$ points picked uniformly in $T$ and $\tilde T$, where we couple the choice of these points in such a way that they fall in subtrees with same label in $T$ and $\tilde T$ (this is completely analogous to the way we couple the branch-points with the random variables $U_\mathbf{i}$ in the previous section).  On the event $\{N_m \le n+1\}$, each branch-point of $\mathcal{S}(m)$ corresponds to a point at which we have glued input trees together.  In particular, the shapes of $\mathcal{S}(m)$ and of the reduced tree $\tilde{\mathcal{S}}(m)$ in $\tilde{T}$ are the same by construction.  Moreover, the lengths of corresponding segments of $\mathcal{S}(m)$ and $\tilde{\mathcal{S}}(m)$ are all made up of sums of scaled distances between pairs of uniform points in trees with the same labels at level $n+1$ and the same scaling factors. (Note that these scaling factors receive appropriate biases from the fact that uniform points have/have not fallen into the corresponding trees, but this affects only the scaling factors and not the underlying trees  since, by construction, the trees and scaling factors are independent.)  By Proposition~\ref{prop:distance}, these distances have the same law in $T$ and $\tilde{T}$.  The result follows since $\epsilon$ was arbitrary and $\tilde{\mathcal{S}}(m) \equidist \mathcal{R}(m)$.
\end{proof}

In view of Theorem~\ref{th:rfdds}, 	Theorem~\ref{th:fxpt} follows.

\section{Convergence to the fixed point: proof of Theorem~\ref{th:cvg}}

\noindent Recall that $M$ is now an arbitrary law on continuum trees. For $n \ge 0$, let $(T_n, d_n, \mu_n) \sim \mathfrak{F}^n M$ and, conditionally on $(T_n, d_n, \mu_n)$, let $V^n_1, V^n_2, \ldots$ be i.i.d.\ points of $T_n$ sampled according to $\mu_n$.  Similarly, let $(\mathbb{T},\d,\m) \sim \mathbb{M}$ and, conditionally on $(\mathbb{T},\d,\m)$, let $V_1, V_2, \ldots$ be i.i.d.\ points of $\mathbb{T}$ sampled according to $\m$.  Write $\hat{\mu}_{n}$ for the law of $d_n(V_1^n, V_2^n)$ and $\hat{\m}$ for the law of $\d(V_1,V_2)$ (which is, of course, Rayleigh).  Let $\mathcal{S}_n(m)$ be the reduced tree of $T_n$ spanned by $V_1^n, \ldots, V_m^n$ and $\mathcal{R}(m)$ be the reduced tree of $\mathbb{T}$ spanned by $V_1, \ldots, V_m$.  Convergence in the Gromov--Prokhorov distance is then equivalent to the convergence 
\[
\mathcal{S}_n(m) \convdist \mathcal{R}(m) \quad \text{as $n \to \infty$}
\]
for each $m \ge 2$ (see Greven, Pfaffelhuber and Winter~\cite{GrevenPfaffelhuberWinter}, or the introduction to Bertoin and Miermont~\cite{BertoinMiermont}). 

We will again use the coupling of Subsection~\ref{sub:coupling} to prove this.  
 Indeed, for fixed $m \ge 3$, we must look to recursion depth $N_m$ in order to separate our $m$ uniform points.   For fixed $\epsilon > 0$, by Proposition~\ref{prop:geom}, we can find $k$ sufficiently large that $\Prob{N_m \le k} > 1 - \epsilon$.  We work on the event $\{N_m \le k\}$. Then, for $n \ge k$, in order to obtain coupled trees distributed as $\mathbb{M}$ and $M_n$ respectively, we need to ``plug in'' $3^{k}$ input trees at level $k$ in the coupling, sampled according to $\mathbb{M}$ and $M_{n-k}$, respectively.  Moreover, as in the proof of Proposition~\ref{prop:redtree}, the lengths of the edges of the reduced trees can then be viewed as sums of scaled distances between uniform points in these trees with distributions $\mathbb{M}$ and $M_{n-k}$. So we need to control the distribution $\hat{\mu}_{n-k}$ of the distance between two uniform points in a tree distributed as $M_{n-k}$.  Note that $\hat{\mu}_{n-k} = \mathfrak{F}_{\mathrm{sm}}^{n-k} \hat{\mu}_{0}$, the $(n-k)$-fold iterate of the smoothing transform $\mathfrak{F}_{\mathrm{sm}}$ applied to the law $\hat{\mu}_{0}$ of the distance between two uniformly sampled points of $T_0 \sim M$.  Theorem 2(b) of Durrett and Liggett~\cite{DurrettLiggett} gives conditions under which repeated applications of the smoothing transform yields convergence to a fixed point.  Recall the function $\nu$ from the proof of Proposition~\ref{prop:distance}.  Then the conditions of Durrett and Liggett's theorem are that (a) $\nu$ has its unique zero in $s \ge 0$ at $s=1$ (b) that $\nu'(1) < 0$ and (c) that the law to which we repeatedly apply $\mathfrak{F}_{\mathrm{sm}}$ should have the same mean as the fixed point.  We already checked (a) and (b) in the course of the proof of Proposition~\ref{prop:distance}.  Moreover, by assumption, $\int_0^{\infty} x \hat{\mu}_{0}(\mathrm d x) = \E{d_0(V_1^0,V_2^0)} = \sqrt{\pi/2} = \int_0^{\infty} x \hat{\m}(\mathrm d x)$, so that (c) also holds.  We conclude that, for fixed $k$, we have $\hat{\mu}_{n-k} = \mathfrak{F}_{\mathrm{sm}}^{n-k} \hat{\mu}_{0} \to \hat{\m}$ as $n \to \infty$.

The edge-lengths in $\mathcal{S}_n(m)$ can then be written as sums of randomly rescaled independent random variables sampled from $\hat{\mu}_{n-k}$. It is then clear (since we use the same random scaling factors in order to construct both) that the edge-lengths of $\mathcal{S}_n(m)$ converge in distribution to those of $\mathcal{R}(m)$ on the event $\{N_m \le k\}$ for any fixed $k \ge 1$.  Since $\epsilon > 0$ was arbitrary, the result follows. \hfill $\Box$

\section{Concluding remarks}
\subsection{Related work}
\noindent As mentioned in Subsection~\ref{sub:BCRTFix}, Aldous~\cite{AldousRecursiveSelfSimilarity} shows that, in a sense, we can ``reverse'' the operator $\fix$. Indeed, we can decompose a BCRT by picking three uniform points and splitting at the branch-point between them; we obtain three independent BCRT's, Brownian-rescaled by $(\Delta_{1},\Delta_{2},\Delta_{3})\sim \mathrm{Dir}(1/2,1/2,1/2)$. Each of these subtrees is doubly marked, one mark being the original uniform point and the other being the former branch-point.  Perhaps a more natural way of phrasing the reversal, which yields only a single mark in each subtree, would be to pick each of the branch-points in the tree with probability given by 6 times the product of the masses of the subtrees into which removal of that branch-point splits the tree.

If we \emph{do} use three uniforms to pick the branch-point then the two marks in each subtree are independent uniform picks from that subtree. This decomposition operation is used recursively by Croydon and Hambly~\cite{CroydonHambly}  to prove that the BCRT is homeomorphic to a certain deterministic fractal with a random self-similar metric, along with the naturally-associated measure. In the course of their proof, they show (Lemma~10(d) of \cite{CroydonHambly}) that all of the randomness in the BCRT is contained in an i.i.d.\ family of $\mathrm{Dir}(1/2,1/2,1/2)$ scaling factors $(\Delta_{\mathbf{i}},\mathbf{i}\in \Sigma)$.

Although we have referred to this decomposition of the BCRT as the reverse of our operator $\fix$, there is, in fact, a rather subtle difference which arises concerning marking and labelling.  The forward version of Croydon and Hambly's splitting operator acts on \emph{doubly uniformly marked trees} and can be paraphrased as follows: take three independent BCRT's, $T_1, T_2, T_3$, each with two independent uniform points, labelled 1 and 2.  Rescale these trees according to the appropriate Dirichlet random vector and glue them together at the points labelled 1.  Now relabel the point labelled 2 in $T_1$ by 1, keep the point labelled 2 in $T_2$ and forget the point labelled 2 in $T_3$ as well as the branch-point just created.  Then this is again a doubly uniformly marked BCRT.  This seems to us a much less natural ``forward'' operation on continuum trees than the one pursued in this paper, but it has the advantage that the recursive decomposition obtained by going backwards does not have any of the labelling issues encountered in Section~\ref{sub:coupling}.  Indeed, in this version there is no randomness in which subtree attaches to which other subtree.

\subsection{Convergence}
\noindent The distributional convergence in Theorem~\ref{th:cvg} is in the sense of the Gromov--Prokhorov distance which, for example, does not distinguish between the BCRT and the BCRT decorated by the independent PPP discussed after Theorem~\ref{th:fxpt}. In particular, this convergence is equivalent to the convergence in distribution of the random finite dimensional distributions. It would be interesting to find conditions under which the convergence holds instead in the stronger Gromov--Hausdorff--Prokhorov sense; in particular, we would need a certain tightness condition to hold (see Corollary 19 of \cite{AldousCRT3}).

\section*{Acknowledgments}
\noindent The question answered in this paper was (to the best of our knowledge) first raised by Nicolas Curien at the YEP VII workshop at EURANDOM in March 2010.  We would like to thank Louigi Addario-Berry and Luc Devroye for inviting us to the Fifth Annual Workshop on Probabilistic Combinatorics and WVD at the Bellairs Institute of McGill University, Barbados, where we began thinking about it, and the Isaac Newton Institute in Cambridge for its invitation in March-April 2015 which enabled us to complete the paper.  We would also like to thank David Croydon for detailed discussions relating to the paper~\cite{CroydonHambly} and Ralph Neininger for discussions about the contraction method. We are grateful to the referee for an extremely thorough reading of the paper which led to considerable improvements in the exposition. C.G.'s research was supported in part by EPSRC Postdoctoral Fellowship EP/D065755/1 and EPSRC grant EP/J019496/1. M.A.\ acknowledges the support of the ERC under the agreement ``ERC StG 208471 - ExploreMap" and the ANR under the agreement ``ANR 12-JS02-001-01 - Cartaplus''.  

\bibliographystyle{amsplain}

\providecommand{\bysame}{\leavevmode\hbox to3em{\hrulefill}\thinspace}
\providecommand{\MR}{\relax\ifhmode\unskip\space\fi MR }
\providecommand{\MRhref}[2]{%
  \href{http://www.ams.org/mathscinet-getitem?mr=#1}{#2}
}
\providecommand{\href}[2]{#2}

\end{document}